\documentclass[reqno]{amsart}
\usepackage{amssymb}
\numberwithin{equation}{section}
\newtheorem{theorem}{Theorem}[section]
\newtheorem{lemma}{Lemma}[section]

\newtheorem{proposition}[lemma]{Proposition}

\def\g{\mathfrak{g}}

\def\l{\mathfrak{l}}

\def\s{\mathfrak{s}}

\newcommand{\R}{\mbox{$\Bbb R$}}

\begin{document}

\title{Lie symmetries of the canonical connection: codimension one abelian nilradical case}
\maketitle

\begin{center}
 {\bf Hassan Almusawa$^1$, Ryad Ghanam$^{2}$, G. Thompson$^3$}\\

$^1$ Department of Mathematics, Jazan University\\
Jazan 45142, Saudi Arabia\\
haalmusawa@jazanu.edu.sa \\
$^2$ Department of Liberal Arts \& Sciences, Virginia Commonwealth University in Qatar\\
PO Box 8095, Doha, Qatar\\
raghanam@vcu.edu\\
$^2$ Department of Mathematics, University of Toledo\\ Toledo, OH
43606, U.S.A. \\  gerard.thompson@utoledo.edu \\

\end{center}

\begin{abstract}
This paper studies the canonical symmetric connection $\nabla$ associated to any Lie group $G$. The salient
properties of $\nabla$ are stated and proved. The Lie symmetries
of the geodesic system of a general linear connection are formulated. The results are then applied to
$\nabla$ in the special case where the Lie algebra $\g$ of $G$, has a codimension one abelian nilradical.
The conditions that determine a Lie symmetry in such a case are completely integrated. Finally the results
obtained are compared with some four-dimensional Lie groups whose Lie algebras have three-dimensional abelian
nilradicals, for which the calculations were performed by MAPLE.
\end{abstract}

\bigskip


\bigskip

AMS: Subject classification: 17B30, 22E15 ,22E60, 53C22

\newpage

\section{introduction}

This article concerns the canonical
symmetric connection $\nabla$ associated to any Lie group $G$.
It was introduced originally in 1926 as the ``zero"-connection; see \cite{CS}.
Some properties of $\nabla$ are considered in \cite{Hel}. More recently
$\nabla$ and its geodesic system have been investigated within the context of the inverse
problem of Lagrangian dynamics; see \cite{GTM, MT, ST}. See also \cite{Pr} for related developments in dynamics. A different direction is taken in
\cite{PA}, which explores the role of $\nabla$ in statistics and information science.

 In the first part of the paper in Section 2 we state and prove all the important geometric
properties of $\nabla$. The second part of the paper is devoted to an examination
of the Lie symmetries of the geodesic system determined by $\nabla$; more specifically,
we consider $\nabla$ in the special case where the Lie algebra $\g$ of $G$ has a codimension
one abelian nilradical. Within that class, we shall need to make further genericity assumptions,
that we explain as we encounter them, in order to obtain some precise results.

The second part of the paper may be viewed as examination of the nature of Lie symmetries themselves.
The cases considered possess a rich class of symmetries and it is of interest to compare them with, for
example, a free particle system whose symmetries we review very briefly in Section 3. As basic references
on Lie symmetries we cite \cite{A,Hy,Ol}. We mention also \cite{ADLT, GT} as concrete examples
of Lie symmetries for several well known differential equations.
The  four-dimensional Lie algebras used in Section 6 are taken from \cite{SW}.

One of the difficulties in executing a package such as MAPLE, that we use in Section 6, is that typically the differential equations
involve parameters; such programs are not sensitive to special values that may break the symmetry of the system. Accordingly, it is necessary to
integrate the conditions determining the symmetries. Along the way, the special values of the parameters should be revealed. In the case at hand,
we are able to carry through this analysis completely in Section 6. The summation convention on repeated indices applies throughout the paper.

\section{ The canonical  Lie group connection}\label{sec4}

On left invariant vector fields $X$ and $Y$ the canonical
symmetric connection $\nabla$ on a Lie group $G$ is defined by
\begin{equation}\label{can} \nabla_XY=\frac{1}{2} \  [X,Y] \end{equation}
\noindent and then extended to arbitrary vector fields using
linearity and the Leibnitz rule. Clearly $\nabla$ is left-invariant. One could just as well use right-invariant vector fields to define $\nabla$, but one must check that
$\nabla$ is well defined, a fact that we will prove next.

\begin{proposition}\label{def}
In the definition of $\nabla$ we can equally assume that $X$ and $Y$ are right-invariant vector fields and hence $\nabla$ is also left-invariant and hence bi-invariant. Moreover
$\nabla$ is symmetric, that is, its torsion is zero.
\end{proposition}
\begin{proof}
The fact that $\nabla$ is symmetric is obvious from eq.(\ref{can}).
Now we choose a fixed basis in the tangent space at the identity $T_IG$. We shall denote its left and right invariant extensions by $\{X_1,X_2,...,X_n\}$ and $\{Y_1,Y_2,...,Y_n\}$,
respectively. Then there must exist a non-singular matrix $A$ of functions on $G$ such that $Y_i=a^j_iX_j$.
We shall suppose that
\begin{equation}\label{struc}
[X_i,X_j]=C_{ij}^kX_k.
\end{equation}
Changing from the left-invariant basis to the right gives
\begin{equation}\label{change}
C^k_{ij}a^p_k=a_i^ka_j^mC^p_{km}.
\end{equation}
Next, we use the fact that left and right vector fields commute to deduce that
\begin{equation}\label{dir}
a_j^kC^m_{ik}+X_ia^m_j=0,
\end{equation}
where the second term in \ref{dir} denotes directional derivative.
We note that necessarily
\begin{equation}\label{minus}
[Y_i,Y_j]=-C_{ij}^kY_k.
\end{equation}
Now we compute
\begin{equation}\label{dire}
\nabla_{Y_i}Y_j+\frac{1}{2}C_{ij}^kY_k=\frac{1}{2}a_i^ka_j^mC^p_{km}+a_i^k(X_ka_i^p)+\frac{1}{2}C_{ij}^ka_k^p.
\end{equation}
Next we use \ref{dir} to replace the second term on the right hand side of \ref{dire} so as to obtain
\begin{equation}\label{diree}
\nabla_{Y_i}Y_j+\frac{1}{2}C_{ij}^kY_k=\frac{1}{2}a_i^ka_j^mC^p_{km}-a_i^ka_j^mC_{km}^p+\frac{1}{2}C_{ij}^ka_k^p.
\end{equation}
However, the right hand side of \ref{diree} is seen to be zero by virtue of \ref{change}. Thus
\begin{equation}
\nabla_{X}Y=\frac{1}{2}[X,Y]
\end{equation}
whenever $X$ and $Y$ are right invariant vector fields.
\end{proof}

An alternative proof of Proposition (\ref{def}) uses the inversion map $\psi$ defined by, for $S\in G$,
\begin{equation}
\psi(S)=S^{-1}.
\end{equation}
As such, one checks that $\psi_{*_I}$ maps a left-invariant vector field evaluated at $I$ to minus its right-invariant counterpart evaluated at $I$.
Then $\psi_{*_I}$ is an isomorphism and there is no change of sign in the structure constants, as compared with eq.(\ref{minus}). Since there are two
minus signs in eq.(\ref{can}) the same condition  eq.(\ref{can}) applies also to right-invariant vector fields.

\begin{proposition}\label{form}
(i) An element in the center of $\g$ engenders a bi-invariant vector field.\\
(ii) A vector field in the center of $\g$ is parallel.\\
(iii) A bi-invariant differential $k$-form $\theta$ is closed and so defines an element of the cohomology group $H^k(M,\R)$.
\end{proposition}
\begin{proof}
(i) Suppose that $Z\in T_IG$ is in the center of $\g$ and let $\exp(tZ)$ be the associated one-parameter subgroup of $G$ so that $Z$ corresponds to the equivalence class of curves $[\exp(tZ)]$ based at $I$. Let $S\in G$; then ${L_S}_*Z$ corresponds to the equivalence class of curves $[S\exp(tZ)]$ based at $S$. Since $Z$ is in the center of $\g$ then $\exp(tZ)$ will be in the center of
$G$ and hence $[S\exp(tZ)]=[\exp(tZ)S]$. It follows that any element in the center of $\g$ engenders a bi-invariant vector field.\\
(ii) Obvious from eq.(\ref{can}).\\
(iii) A proof can be found in \cite{Spi}. Spivak shows that $\psi^*(\theta)=(-1)^k\theta$, whereas $d\theta$, which is also bi-invariant, changes by
$\psi^*(d\theta)=(-1)^{k+1}d\theta$. It follows that $d\theta=0$.

\end{proof}

\begin{proposition}
(i) The curvature tensor, which is also bi-invariant, on vector fields $X,Y,Z$ is given by
\begin{equation} R(X,Y)Z =\frac{1}{4} \ [[X,Y],Z]. \label{curv} \end{equation}\\
(ii) The connection $\nabla$ is flat if and only if the Lie algebra $\g$ of $G$ is two-step nilpotent.\\
(iii) The tensor $R$ is parallel in the sense that $\nabla_W R(X,Y)Z=0$, where $W$ is a fourth right invariant vector field, so
that $G$ is in a sense a symmetric space.\\
(iv) The Ricci tensor $R_{ij}$ of $\nabla$ is given by
\begin{equation} R_{ij}=\frac{1}{4} \ C^l_{jm} C^m_{il} \label{Ric} \end{equation}
and is symmetric and bi-invariant and is obtained by translating to the left or right one
quarter of the Killing form. It engenders a bi-invariant pseudo-Riemannian metric if and only if the  Lie
algebra $\g$ is semi-simple.
\end{proposition}
\begin{proof} (i) Is obvious and applies to arbitrary vector fields since it is a tensorial object.\\
(ii) Is obvious.\\
(iii) This fact follows from a series of implications:
\begin{eqnarray}
4\nabla_W R(X,Y)Z+4R(\nabla_W X,Y)Z+4R(X,\nabla_W Y)Z+4R(X,Y)\nabla_W Z=\nabla_W [[X,Y],Z]\\
\implies 4\nabla_W R(X,Y)Z+2R([W,X],Y)Z+2R(X,[W,Y])Z+2R(X,Y)[W,Z]-\frac{1}{2}[W,[[X,Y],Z]=0\\
\implies 4\nabla_W R(X,Y)Z+\frac{1}{2}[[W,X],Y],Z]+\frac{1}{2}[X,[W,Y]],Z]+\frac{1}{2}[[X,Y],[W,Z]]-\frac{1}{2}[W,[[X,Y],Z]=0\\
\implies 4\nabla_W R(X,Y)Z+\frac{1}{2}[[W,X],Y],Z]+\frac{1}{2}[X,[W,Y]],Z]-\frac{1}{2}[Z,[[X,Y],W]=0\\
\implies \nabla_W R(X,Y)Z=0.
\end{eqnarray}
(iv) The formula eq.(\ref{Ric}) is obvious from eqs. (\ref{can}) and (\ref{curv}). The last remark follows from Cartan's criterion.
\end{proof}

\begin{proposition}
(i) Any left or right-invariant vector field is geodesic.\\
(ii) Any geodesic curve emanating from the identity is a one-parameter subgroup..\\
(iii) An arbitrary geodesic curve is a translation, to the left or right, of a one-parameter subgroup.
\end{proposition}
\begin{proof} (i) Is obvious because of the skew-symmetry in eq.(\ref{can}).\\
(ii) By definition the curve $t\mapsto [S\exp(tX)]$ integrates a geodesic field $X$.\\
(iii) If the geodesic curve at $t=0$ starts at $S$, translate the curve to $I$ by multiplying on the left or right by $S^{-1}$ and apply (ii).
\end{proof}

\begin{proposition}
(i) A left or right-invariant vector field is a symmetry, a.k.a. affine collineation, of $\nabla$.\\
(ii) Any left or right-invariant one-form engenders a first integral of the geodesic system of $\nabla$.\\
\end{proposition}
\begin{proof} (i) The following condition for vector fields $X$ and $Y$ says that vector field $W$ is a symmetry or, affine collineation, of a symmetric linear connection:
\begin{equation}\label{coll}
\nabla_X \nabla_Y W- \nabla_{\nabla_X Y} W - R(W,X)Y=0.
\end{equation}
In the case at hand of the canonical connection, this condition just reduces to the Jacobi identity when $W,X$ and $Y$ are all left or right-invariant.\\
(ii) A one-form $\alpha$ is a \emph{Killing one-form}, if the following condition holds:
\begin{equation}\label{kil}
\langle{\nabla_X \alpha,Y}\rangle + \langle{X,\nabla_Y\alpha}\rangle=0.
\end{equation}
In the  case of the canonical connection, if $X$ and $Y$ are right-invariant and $\alpha$ is right-invariant then eq.(\ref{can}) gives
\begin{equation}\label{kill}
\langle{X,\nabla_Y\alpha}\rangle =\frac{1}{2}\langle{[X,Y],\alpha}\rangle .
\end{equation}
Clearly, (\ref{kill}) implies (\ref{kil}) so that every left or right-invariant one-form engenders a first integral of the geodesics: if the one-form is given in a coordinate system
as $\alpha_i dx^i$ on $G$, the first integral is $\alpha_{i}u^{i}$ viewed as a function on the tangent bundle $TG$ that is linear in the fibers.
\end{proof}

\begin{proposition}\label{ann}
Any left or right-invariant one-form $\alpha$ is closed if and only if $\langle{[\g,\g],\alpha}\rangle=0$, that is, $\alpha$
annihilates the derived algebra of $\g$.
\end{proposition}
\begin{proof}
Consider the identity
\begin{equation}\label{closed}
d\alpha(X,Y)= X\langle{Y,\alpha}\rangle-Y\langle{X,\alpha}\rangle-\langle{[X,Y],\alpha}\rangle.
\end{equation}
If $\alpha$ is left-invariant and we take $X$ and $Y$ left-invariant, then the first and second terms in eq.(\ref{closed})
are zero. Now the conclusion of the Proposition is obvious. The proof for right-invariant one-forms is similar.
\end{proof}

\begin{proposition}Consider the following conditions for a one-form $\alpha$ on $G$:

{\rm (i)} $\alpha$ is bi-invariant.

{\rm (ii)}  $\alpha$ is right-invariant and closed.

{\rm (iii)} $\alpha$ is left-invariant and closed.

{\rm (iv)} $\alpha$ is parallel.

\noindent Then we have the following implications: (i), (ii) and (iii) are equivalent and any one of them implies (iv).
\end{proposition}
\begin{proof} The fact that (i) implies (ii) and (iii) follows from Proposition (\ref{form}) part (iii). Now suppose that (iii) holds and let $X$ and $Y$ be right and left-invariant vector fields, respectively.
Then consider again the identity
\begin{equation}\label{closedd}
d\alpha(X,Y)= X\langle{Y,\alpha}\rangle-Y\langle{X,\alpha}\rangle-\langle{[X,Y],\alpha}\rangle.
\end{equation}
Assuming that $\alpha$ is closed, then either because $[X,Y]=0$ or by using Proposition \ref{ann}, we find that eq.(\ref{closedd}) reduces to
\begin{equation}\label{closeddd}
X\langle{Y,\alpha}\rangle=Y\langle{X,\alpha}\rangle.
\end{equation}
Now the left hand side of eq.(\ref{closeddd}) is zero, since $Y$ and $\alpha$ are left-invariant. Hence $\langle{X,\alpha}\rangle$ is
constant, which implies that $\alpha$ is right-invariant and hence bi-invariant. Thus (iii) implies (i). The proof that
(ii) implies (i) is similar.
Finally, supposing that (ii) or (iii) holds we show that (iv) holds. Then as with any symmetric connection, the closure condition may be written,
for arbitrary vector fields $X$ and $Y$, as
\begin{equation}\label{killl}
\langle{\nabla_X \alpha,Y}\rangle - \langle{X,\nabla_Y\alpha}\rangle=0.
\end{equation}
Clearly eq.(\ref{kil}) and eq.(\ref{killl})imply that $\alpha$ is parallel. So a closed, invariant one-form is parallel.
\end{proof}

Of course, it may well be the case that there are no bi-invariant one-forms on $G$, for example if $G$ is
semi-simple so that $[\g,\g]=\g$. However, there must be at least one such one-form if $G$ is solvable and at least two if $G$ is nilpotent.

If we choose a basis of dimension dim $\g$-dim $[\g,\g]$ for the bi-invariant one-forms on $G$, it may be used to obtain a partial coordinate system on
$G$, since each such form is closed. Such a partial coordinate system is significant in terms of the geodesic system, in that it
gives rise to second order differential equations that resemble the system in Euclidean space.

\begin{proposition}
Each of the bi-invariant one-forms on $G$ projects to a one-form on the quotient space $G/[G,G]$, assuming that the
commutator subgroup $[G,G]$ is closed topologically in $G$. Furthermore the canonical connection $\nabla$ on $G$ projects to
a flat connection on $G/[G,G]$ and the induced system of one-forms on $G/[G,G]$ comprises a ``flat"  coordinate system.
\end{proposition}
\begin{proof} The fact that a bi-invariant one-form on $G$ projects to a one-form on $G/[G,G]$ follows because each such form
annihilates the vertical distribution of the principal right $[G,G]$-bundle $G\rightarrow G/[G,G]$ and furthermore the equivariance, or Lie-derivative condition along
the fibers, is trivially satisfied since the one-form is closed. The fact that $\nabla$ projects to $G/[G,G]$ follows because  $[G,G]\triangleleft G$, as was noted in \cite{KT}.
\end{proof}

\section{Lie symmetries of a free particle}

In this short Section we recall the Lie symmetry vector fields of a free particle system in dimension $n$.
It may be regarded as the extreme case of the geodesic system of the canonical connection for
an \emph{abelian} Lie group. The geodesics are given by
\begin{equation}
\label{conn} \ddot{x}^i = 0.
\end{equation}

\noindent Lie symmetries:
\begin{equation}D_t, D_{x^i}, tD_t, x^iD_t, tD_{x^i}, x^iD_{x^j}, t\Delta, x^i\Delta,\end{equation}
\noindent where $\Delta=tD_t+x^iD_{x^i}$ and gives $\s\l(n+2,\R)$ as symmetry Lie algebra. We note that
$\s\l(n+2,\R)$ is of dimension $(n+2)^2-1=(n+3)(n+1)$.

\section{codimension one abelian nilradical}

\subsection{Coordinate normal form}

Rather than studying the canonical connection for general Lie groups, we will now examine a particular class, namely, those groups for which
the associated Lie algebras are solvable and have a codimension one abelian nilradical.
The following theorem gives a coordinate normal form for the
canonical connection of a Lie algebra that has a codimension one
abelian ideal. Such Lie algebras are characterized by a matrix
ad$(e_{n+1})$ where $e_{n+1}$ is a fixed element of $\g$, usually taken as
the last element of a basis.

\begin{theorem}\label{Thm} {\rm (i)} Suppose that $\g$ is Lie algebra of dimension $n+1$ such that
there exists a basis $\{e_1,e_2,...,e_n,e_{n+1}\}$ for which the
only non-zero brackets in $\g$ are given by $[e_i,e_{n+1}] =
a_i^je_j$, where, on the right hand side the summation over $j$
extends from $1$ to $n$. Then there exists a coordinate system
$(x^i,w)$ on the local Lie group $G$ associated to $\g$, such that
$\g$ is faithfully represented by $(X_i,W)$ where $X_i =
\frac{\partial}{\partial x^i}$ and $W= \frac{\partial}{\partial
w}+a^k_jx^j\frac{\partial}{\partial x^k}$.

\noindent {\rm (ii)} In the coordinate system $(x^i,w)$ of  (i) the geodesic
equations of the canonical connection are given by
\begin{equation}
\label{conn} \ddot{x}^i = a^i_j\dot{x}^j\dot{w},\quad \ddot{w} = 0.
\end{equation}
\end{theorem}

\bigskip
\noindent Part (i) of Theorem \ref{Thm} gives an
automatic implementation of Lie's third theorem \cite{Hel}. Part
(ii) then gives a normal form for the geodesics of the canonical
connection. In general, these geodesics can only be computed if a
vector field representation of $\g$ is known and only then by
means of considerable calculation.

Although we shall not need it, it is possible to obtain a matrix Lie group representation in $GL(n+1,\R)$ by
exponentiating the matrix $a_j^i$, assuming that it is upper triangular, and taking $n$ arbitrary entries in the last column.

In fact $(X_i,W)$ are the right-invariant vector fields and the dual one-forms are given by
\begin{equation}
dx^i-a_j^ix^jdw,dw.
\end{equation}
Similarly, the left-invariant vector fields are given by
\begin{equation}
e^{wa_j^i}\frac{\partial}{\partial x^i}, \frac{\partial}{\partial w}
\end{equation}
and the dual one-forms are given by
\begin{equation}
e^{-wa_j^i}dx^j, dw.
\end{equation}
We note also that $dw$ is bi-invariant.

\section{Lie symmetries of the Canonical Connection}



We encode the geodesic equations of a connection in terms of the vector field $\Gamma=\frac{\partial}{\partial t}+u^i\frac{\partial}{\partial x^i}+f^i\frac{\partial}{\partial u^i}$. In order to be a Lie symmetry we need a vector field of the form $X=\xi\frac{\partial}{\partial t}+\eta^i\frac{\partial}{\partial x^i}$ such that $[\tilde{X},\Gamma]=\lambda \Gamma$ for some function $\lambda$, where $\tilde{X}$ denotes the first prolongation of $X$. It is apparent that $\frac{\partial}{\partial t}$ is a Lie symmetry. In the case of  $t\frac{\partial}{\partial t}$, its prolongation is given by $t\frac{\partial}{\partial t}-\Delta$ where  $\Delta$  is the dilation field $u^i\frac{\partial}{\partial u^i}$. It follows that $t\frac{\partial}{\partial t}$ is a Lie symmetry, with $\lambda=-1$, because $[\Delta, \Gamma]=\Gamma-\frac{\partial}{\partial t}$. Of course $\frac{\partial}{\partial t}$ and $t\frac{\partial}{\partial t}$ are symmetries for any linear connection. Thus, at the outset, we have the following Lie symmetries for the canonical connection:

\begin{itemize}

\item $\frac{\partial}{\partial t}$

\item $t\frac{\partial}{\partial t}$

\item any left-invariant vector field

\item any right-invariant vector field

\end{itemize}

\subsection{Lie symmetries of a symmetric connection}

Define \begin{equation}
P^i=\dot{\eta}^i-u^i\dot{\xi}
\end{equation}
then
\begin{equation}
\tilde{X}=\xi\frac{\partial}{\partial t}+\eta^i\frac{\partial}{\partial x^i}+P^i\frac{\partial}{\partial u^i}
\end{equation}
and
\begin{equation}
\dot{P^i}=\ddot{\eta}^i-u^i\ddot{\xi}-f^i\dot{\xi}.
\end{equation}
The conditions for $X$ to be Lie symmetry amount to
\begin{eqnarray}
\lambda=-\dot{\xi}\\
\dot{\xi}u^i=\dot{\eta}^i-P^i \\
\dot{\xi}f^i+P^j\frac{\partial f^i}{\partial u^i}-\dot{P^i}+\xi \frac{\partial f^i}{\partial t}+\eta^j \frac{\partial f^i}{\partial x^j}=0  \label{third}
\end{eqnarray}
the second of which is an identity. Using the definition of $P^i$ in \ref{third} gives
\begin{equation} \label{symmcond}
\dot{\xi}f^i=\ddot{\eta}^i-u^i\ddot{\xi}-f^i\dot{\xi}-(\dot{{\eta}^j}-u^j\dot{\xi})\frac{\partial f^i}{\partial u^j}-\eta^i\frac{\partial f^j}{\partial x^i}.
\end{equation}
Finally since the $f^i$ are are homogeneous quadratic $u^i$, we find that eq.(\ref{symmcond}) reduces to
\begin{equation} \label{symmcondd}
\ddot{\eta}^i-u^i\ddot{\xi}-\dot{{\eta}^j}\frac{\partial f^i}{\partial u^j}-\eta^i\frac{\partial f^j}{\partial x^i}=0.
\end{equation}

Next we shall make use of the following identities:
\begin{equation} \label{xicond}
\ddot{\xi}=\xi_{tt}+2u^j\xi_{tk}+u^ju^k\xi_{jk}+f^i\xi_{i}
\end{equation}
\begin{equation} \label{etacond}
\ddot{\eta^i}=\eta^i_{tt}+2u^j\eta^i_{tj}+u^ju^k\eta^i_{jk}+f^j\eta^i_{j}
\end{equation}
where we have used subscripts to denote derivatives. Eq.(\ref{symmcondd}) contains terms of degree  zero, one, two and three in the $u^i$ and we write down the coefficients of each term so as to obtain
\begin{eqnarray}
\label{ge01} \frac{\partial^2 {\eta}^i}{\partial t^2}=0\\
\label{ge02} 2u^j\frac{\partial^2 {\eta}^i}{\partial x^j \partial t}-u^i\frac{\partial^2 \xi}{\partial t^2}-\frac{\partial {\eta}^j}{\partial t}\frac{\partial f^i}{\partial u^j}=0\\
\label{ge03} u^ju^k\frac{\partial^2 \eta^i}{\partial x^j \partial x^k}+f^j\eta^{i}_{j}-{\eta}^j\frac{\partial f^i}{\partial x^j}-2u^iu^j\frac{\partial^2 \xi}{\partial t \partial x^j}-u^k\frac{\partial \eta^j}{\partial x^k}\frac{\partial f^i}{\partial u^j}=0\\
\label{ge04} u^iu^ju^k\frac{\partial^2 \xi}{\partial x^j \partial x^k}+u^if^k\frac{\partial \xi}{\partial x^k}=0.
\end{eqnarray}

Now we shall write
\begin{equation}
f^i=-\Gamma_{jk}^i u^ju^k.
\end{equation}
Then eqs.(\ref{ge02}, \ref{ge03}) and eq.(\ref{ge04}) give, on equating powers of the $u^i$ to zero,
\begin{eqnarray}
\label{ge05} 2\frac{\partial^2 {\eta}^i}{\partial x^k \partial t}-\delta_k^i\frac{\partial^2 \xi}{\partial t^2}+2\frac{\partial \eta^j}{\partial t}\Gamma_{jk}^i=0\\
\label{ge06}\frac{\partial^2 \xi}{\partial x^j \partial x^k}-\Gamma^{i}_{jk}\xi_i=0\\
\label{ge07} \frac{\partial^2 \eta^i}{\partial x^k \partial x^m}+\frac{\partial \eta^j}{\partial x^k}\Gamma_{jm}^i+\frac{\partial \eta^j}{\partial x^m}\Gamma_{jk}^i+ {\eta}^j \frac{\partial {\Gamma}_{km}^i}{\partial x^j}-\delta_k^i\frac{\partial^2 \xi}{\partial t \partial x^m}-\delta_m^i\frac{\partial^2 \xi}{\partial t \partial x^k}-\eta_{j}^{i}\Gamma_{km}^{j}=0.
\end{eqnarray}


\subsection{More Lie symmetries}

In the case of the canonical connection, where the associated Lie algebra has a codimension one abelian nilradical, the connection components are constant and so each of the coordinate vector fields $\frac{\partial}{\partial x^i}$ is a Lie symmetry for which $\lambda=0$ and of course they are right-invariant as noted in Theorem \ref{Thm}. It is apparent that $w\frac{\partial}{\partial t}$ is always a Lie symmetry with $\lambda=- \dot{w}$.

Another Lie symmetry is given by $X=x^i\frac{\partial}{\partial x^i}$. To see that it is so, note that the prolongation is given by $\tilde{X}=x^i\frac{\partial}{\partial x^i}+u^i\frac{\partial}{\partial u^i}$. Then
\begin{eqnarray}
[\tilde{X},\Gamma]=[x^i\frac{\partial}{\partial x^i}+u^i\frac{\partial}{\partial u^i}, \frac{\partial}{\partial t} +u^i\frac{\partial}{\partial x^i}+ \dot{w}\frac{\partial}{\partial w}+f^i\frac{\partial}{\partial u^i}]\\
=[x^i\frac{\partial}{\partial x^i},u^i\frac{\partial}{\partial x^i}]+[u^i\frac{\partial}{\partial u^i},u^j\frac{\partial}{\partial x^j}]+[u^i\frac{\partial}{\partial u^i},a_k^ju^k \dot{w}\frac{\partial}{\partial u^j}]\\=-u^i\frac{\partial}{\partial x^i}+u^i\frac{\partial}{\partial x^i}=0.
\end{eqnarray}







\subsection{Codimension abelian nilradical case}

Next we shall adapt eqs.(\ref{ge01}  \ref{ge05}  \ref{ge06}  \ref{ge07}) to the codimension abelian nilradical case. Note that the Lie algebra is now of dimension $n+1$. The first $n$ coordinates are denoted by $x^i$ and the $(n+1)$th by $w$. We shall denote a Lie symmetry now by
\begin{equation}
X=\xi\frac{\partial}{\partial t}+ \eta^i\frac{\partial}{\partial x^i}+\eta \frac{\partial}{\partial w}.
\end{equation}
and again we shall employ the summation convention in the range $1$ to $n$. The connection components coming from eq.(\ref{conn}) are given by
\begin{equation}
\Gamma_{jk}^i=0,\, \Gamma_{jn+1}^i=\Gamma_{n+1j}^i=-\frac{1}{2}a_j^i,\, \Gamma_{n+1n+1}^i=0,\, \Gamma_{ij}^{n+1}=0,\, \Gamma_{in+1}^{n+1}=\Gamma_{n+1i}^{n+1}=0,\, \Gamma_{n+1n+1}^{n+1}=0.
\end{equation}
We find the following conditions on $\xi,\eta^i, \eta$ where in the interest of brevity, derivatives with respect to $t$ and $w$ are denoted by subscripts and derivatives with respect to $x^i$ is denoted by subscript $i$ and 
\begin{eqnarray}
(i)\, \xi_{jk}=0\, (ii)\, \xi_{iw}+\frac{1}{2}a_i^j \xi_j=0\, (iii) \, \xi_{ww}=0 \, (iv) \eta_{tt}=0\, (v)\, \eta_{tk}=0\, (vi)\, \eta_{km}=0 &&  \, \nonumber \\
(vii) \eta_{kw}-\xi_{tk}+\frac{1}{2}a^{j}_{k}\eta_{j}=0 \,  (viii) \, \eta_{ww}-2\xi_{tw}=0, (ix)\, 2\eta_{tw}-\xi_{tt}=0\, && \nonumber \\
(x)\, \eta_{tt}^i=0\, \, (xi)\, \eta^i_{tw}-\frac{1}{2}a_k^i\eta^k_{t}=0\, (xii)\, \eta^i_{ww}-a_j^i\eta^j_{w}=0 && \nonumber \\
(xiii)\, 2\eta^i_{tk}-\delta_k^i \xi_{tt}-a_k^i\eta_{t}=0\, (xiv)\, \eta^i_{km}-\frac{1}{2}a_m^i\eta_{k}-\frac{1}{2}a_k^i\eta_{m}-\delta_k^i\xi_{tm}-\delta_m^i\xi_{tk}=0 && \,  \nonumber \\ (xv)\, \eta^i_{kw}-\frac{1}{2}a_j^i\eta^j_{k}-\frac{1}{2}a_k^i\eta_{w}+\frac{1}{2}a_k^j\eta_j^i-\delta_k^i\xi_{tw}=0. && \nonumber
\end{eqnarray}

\subsection{Solving the PDE}

Now we integrate {(i)-(xv)} above, to the extent possible. As such the solution to (iv),(v),(vi) is given by
\begin{equation}
\eta=B_k(w)x^k+C(w)t+D(w).
\end{equation}
Turning now to (ii) and (iii), if we take the $w$-derivative of (ii) and use (ii) we conclude that
\begin{equation}
a_i^j \xi_{jw}=0.
\end{equation}
Let us continue by assuming that the matrix $A$ is non-singular. Then $\xi_{jw}=0$ and using (ii) again, we find that $\xi_{j}=0$.
As such, from (iii) we conclude that
\begin{equation}
\xi=E(t)w+F(t).
\end{equation}

Now from (viii) and (ix) we obtain $\xi_{ttw}=0$ and $\eta_{tww}=0$. From the $w$-derivative of (vii)
and the $x^k$-derivative of (viii) we deduce that $\eta_{jw}=0$ and hence again from (vii) that  $\eta_{j}=0$.
Integrating (viii) and (ix) gives
\begin{equation}\label{xieta}
\xi=(Ft+G)w+K+Lt+Ut^2,\,\eta=Ct+Fw^2+Hw+J+Utw.
\end{equation}

Now we consider (x),(xi) and (xii). Take the $w$ derivative of (xi) and the $t$ derivative of (xii). We deduce that $\eta_{tw}^k=0$ and hence
from (xi) that $\eta_{t}^k=0$, since we are assuming that $A$ is non-singular. Concerning (xii), we find that the solution is
\begin{equation}
\eta^i=A^{-1}e^{wA}M^i(x)+N^i(x).
\end{equation}

It remains to examine (xiii),(xiv) and (xv). However, we see that in view of the conditions that have already been solved, (xiii) gives
$C=U=0$.
Furthermore (xiv) reduces to
\begin{equation}
\eta_{km}^i=0.
\end{equation}
Hence we may write
\begin{equation}\label{etai}
\eta^i={(a^{-1})}_j^ie^{wa_k^j}(P_m^k x^m+Q^k)+R_j^i x^j+S^i.
\end{equation}
At this point, only (xv) remains to be satisfied and it is
\begin{equation}\label{new}
e^{wa_j^i}P_k^j+a_k^j({(a^{-1})}_m^ie^{wa_l^m}P_j^l+R_j^i)-a_j^i({(a^{-1})}_m^je^{wa_l^m}P_k^l+R_k^j)-a_k^i(2Fw+H)-\delta_k^iF=0.
\end{equation}
Eq.(\ref{new}) splits into two conditions, the first of which may be written as
\begin{equation}\label{neww}
e^{wA}PA+Ae^{wA}P=0.
\end{equation}
Eq.(\ref{neww}) is equivalent to
\begin{equation}\label{newwww}
AP+PA=0
\end{equation}
that is, $P$ anti-commutes with $A$, and for generic $A$,
will possess only the zero matrix as solution.

The second condition coming from Eq.(\ref{new}) is
\begin{equation}\label{newww}
a_k^jR_j^i-a_j^iR_k^j-(2Fw+H)a_k^i-F\delta_k^i=0.
\end{equation}
Note that there are no conditions on $Q^k$ and $S^k$ and that necessarily $F=0$. After putting $F=0$, eq(\ref{newww}) may be written in matrix form as
\begin{equation}\label{comm}
[R,A]=HA
\end{equation}
where the left hand side is a commutator.
From eq.(\ref{comm}) we conclude that either $H=0$ or tr$(A)=0$. If
tr$(A)=0$ then the Lie algebra is unimodular; if tr$(A)\neq0$ then $R$ commutes with $A$.

To summarize, the solution for $\xi,\eta^i,\eta$ is given by eq(\ref{xieta}) with $C=F=U=0$
and 
\begin{equation}\label{etaii}
\eta^i=R_j^i x^j+S^i+e^{wa_k^j}(P_j^kx^j+T^k)
\end{equation}
where $H$ and $R$ satisfy condition eq.(\ref{comm}) and $P$ satisfies eq.(\ref{newwww}).

\subsection{Analyzing the solutions of the algebraic condition \ref{comm}}

From the preceding analysis we see that the only ambiguity that arises in the determination of the possible Lie symmetries depends on the solution
of the \emph{algebraic} eq.( \ref{comm}). We are already assuming that the matrix $A$ is non-singular. If furthermore, the codimension one Lie algebra
is not unimodular, that is traces of all ad-matrices are not zero, so in our case tr$A\neq0$, then necessarily $H=0$ and $R$ must commute with $A$.
The centralizer of $A$ in $\g\l(n,\R)$ certainly contains all polynomials in $A$ and equality holds if and only if $A$ is \emph{non-derogatory}. In fact $A$ is non-derogatory
if and only if its minimum polynomial coincides with its characteristic polynomial if and only if each eigenvalue is of geometric multiplicity one. The set of non-derogatory
matrices is Zariski open and dense and contains the set of all matrices with distinct eigenvalues. For such matrices the centralizer of $A$ is of dimension $n$;
otherwise, the dimension of the centralizer of $A$ becomes more problematic to estimate. The algebra of polynomials can have quite a small dimension. The extreme case of all
is when $A$ is a non-zero multiple of the identity, in which case all matrices in $\g\l(n,\R)$ commute with $A$. This case corresponds to algebras $A_{3.3}, A_{4.5(a=b=1)}, A_{5.7(a=b=c=1)}$
in \cite{PSWZ}.

\subsection{Analyzing the solutions of the PDE}

It is apparent that the coefficients of $S^i$ and $e^{wa_k^j}T^k$ in \ref{etaii} correspond to right and left-invariant vector fields.
If we take $R=A$ and $H=0$ in eq.(\ref{comm}) we obtain $a_j^ix^j\frac{\partial }{\partial x^i}$ as a symmetry vector field. If we add to it $\frac{\partial }{\partial w}$
coming from $J$ in eq(\ref{xieta}) we obtain the $n+1$st right-invariant vector field. Note also that $\frac{\partial }{\partial w}$ is the $n+1$st left-invariant vector field.

One more symmetry vector field may be obtained by taking $R$ as the identity and again $H=0$ in eq.(\ref{comm}). This vector field was already found in Section 5.2. Similarly
from $\xi$ appearing in  eq(\ref{xieta}) we reproduce $\frac{\partial }{\partial t}, t\frac{\partial }{\partial t}$ and $w\frac{\partial }{\partial t}$.

Finally we count the minimal number symmetry vector fields. We have $n+1$ right-invariant and $n+1$ left-invariant vector fields. We have also
$\frac{\partial }{\partial t}, t\frac{\partial }{\partial t}$ and $w\frac{\partial }{\partial t}$. In the generic case, where $A$ is non-singular, has trace non-zero
is non-derogatory and the only solution to eq.(\ref{newwww}) is trivial, we obtain an extra $n-1$ independent symmetries coming from eq.(\ref{comm}) and also $H=0$. At the other extreme, where $A$ is a non-zero multiple
of the identity, we obtain $n^2+2n+4$ independent symmetries.

\begin{theorem}\label{thm}
In case the matrix $A$ is non-singular and has trace non-zero the dimension of the Lie symmetry algebra of the
geodesic system eq.(\ref{conn}) is at least $3n+4$
and at most $n^2+2n+4$; moreover, $3n+4$ occurs only if $A$ is non-derogatory and $n^2+2n+4$
if and only if $A$ is a non-zero multiple of the identity.
\end{theorem}




\section{Four-dimensional canonical geodesic systems and their symmetry algebras}

In this final Section we illustrate our results by considering Lie symmetries of the canonical connection
with codimension one abelian nilradical case in dimension four. The calculations here have been performed using
MAPLE. The algebras correspond to cases $A_{4.2a},A_{4.3},A_{4.4},A_{4.5ab},A_{4.6ab}$
in \cite{SW}. Our analysis is extensive although not exhaustive for these five cases. The matrix $A$ that played the key role in Sections
4 and 5 here is $3\times3$. In case $A_{4.3}$ the matrix $A$ is singular, so our results are not applicable; nonetheless, it is included
as it suggests that the symmetry algebra will be of greater dimension when $A$ is singular. In all the other cases, the symmetry algebra
is of dimension at least $13$ in accordance with Theorem (\ref{thm}); in several cases extra symmetries occur because there are non-trivial solutions to eq.(\ref{newwww}).

\subsection{$A_{4.2a}$}

Lie Brackets: $[e_1,e_4]=ae_1, [e_2,e_4]=e_2+e_3, [e_3,e_4]=e_3,\,\,(a \neq 0)$.

\noindent Geodesics: \begin{equation} \begin{array}{l}
 \ddot{x}=a\dot{x}\dot{w} \\
 \ddot{y}=(\dot{y}+\dot{z})\dot{w}\\
 \ddot{z}=\dot{z}\dot{w} \\
 \ddot{w}=0.
\end{array} \end{equation}

\noindent Lie symmetries $a\neq \pm 1$:
\[ \begin{array}{l} e_{1}=-zD_y,\,e_{2}=D_z,\,\,\,e_{3}=e^w(D_z+ wD_y)\,,e_{4}=D_y,\,\,, e_{5}=e^wD_y,\,\,\,e_6=wD_t, e_7=D_t,\,\,\,\\
e_8= D_x, \,\,\,
e_9=e^{aw} D_x,\,\,
 e_{10}=yD_y+zD_z,\,\,\, e_{11}=-D_w, e_{12}=tD_t,\,\,\,
e_{13}=xD_x.
\end{array} \]


\noindent The Lie algebra $\R^4 \rtimes (H_{5}\oplus \R^4)$ is $13$-dimensional solvable. It has a
$9$-dimensional non-abelian nilradical $H_5\oplus \R^4$. Here $H_5$ denotes the $5$-dimensional Heisenberg algebra and is  spanned by $e_1, e_2, e_3, e_4, e_5$ and the $\R^4$ summand is spanned by $e_6, e_7, e_8, e_9$. The $4$-dimensional abelian complement to $H_5\oplus \R^4$ is spanned by $e_{10}, e_{11}, e_{12}, e_{13}$.\\

\noindent Lie symmetries $a=\pm 1$:
\[ \begin{array}{l}
e_1= D_x, \,\,\, e_2=e^{aw}D_x,\,\,\, e_{3}=zD_y,\,\,\, e_{4}=D_z, e_{5}=e^w(D_z+wD_y),\,\,\, e_{6}=D_y,\,\,\,e_{7}=e^wD_y,\,\,\,\\
e_{8}=xe^{\frac{(1-a)w}{2}}D_y, \,\,\, e_{9}=ze^{\frac{(a-1)w}{2}}D_x,\,\,\,e_{10}=wD_t, e_{11}=D_t,\,\,\,e_{12}= yD_y+zD_z,\,\,\, e_{13}=D_w,\\
e_{14}=tD_t,\,\,\,  e_{15}=xD_x.\\
\end{array}\]\\
\noindent In each case the Lie algebra $\R^4 \rtimes (N_9\oplus \R^2)$ is $15$-dimensional solvable. It has
an $11$-dimensional decomposable nilradical, $N_9\oplus \R^2$, where  $N_{9}$ is is $9$-dimensional  indecomposable nilpotent spanned by $e_1,e_2,e_3,e_4,e_5, e_6, e_7,
e_8, e_9$ and $\R^2$ by $e_{10}, e_{11}$ and a $4$-dimensional abelian
complement spanned by $e_{12}, e_{13}, e_{14}, e_{15}$. The ``extra" symmetry for $a=1$ is explained by the fact that only for that value is the
matrix $A$ not non-derogatory; meanwhile $a=-1$ is the unique value for which  eq.(\ref{newwww}) has a non-trivial solution.


\subsection{$A_{4.3}$}

Lie brackets: $[e_1,e_4]=e_1, [e_3,e_4]=e_2$.

\noindent Geodesics:
\begin{equation}
\begin{array}{l}
 \ddot{x}=\dot{x}\dot{w} \\
 \ddot{y}=\dot{z}\dot{w}\\
 \ddot{z}=0 \\
 \ddot{w}=0.
\end{array}
\end{equation}

\bigskip

\noindent Lie symmetries:

\[ \begin{array}{l}
e_1= wD_t,\,\,\, e_2=\frac{w^2}{2}D_y+wD_z,\,\,\, e_3=e^wD_x,\,\,\,  e_4=D_y, e_5=D_x, \,\,\, e_6=D_z,
e_7=D_t,\,\,\, e_8= wD_y, \,\,\,\\ e_9=tD_t+yD_y+zD_z,\,\,\, e_{10}=D_w+\frac{z}{2}D_y,\,\,\, e_{11}=xD_x, \,\,\, e_{12}=(wz-y)D_y+zD_z\\
e_{13}=\frac{1}{3}(2tD_t-yD_y-zD_z),\,\,\, e_{14}=tD_y,\,\,\,e_{15}=zD_t,
e_{16}=\frac{tw}{2}D_y+tD_z,\,\,\, e_{17}= (wz-2y)D_t,\,\,\,\\ e_{18}= zD_y,\,\,\,
e_{19}=(\frac{zw^2}{2}-yw)D_y +(wz-2y)D_z.\\
\end{array} \]\\
\noindent The symmetry algebra is $\s\l(3,\R) \rtimes ( \R^3\rtimes \R^8)$ where $\s\l(3,\R)$ is spanned by $e_{12},e_{13},e_{14},e_{15},e_{16},e_{17},e_{18},e_{19}$, the $\R^3$ factor is spanned by $e_{9},e_{10},e_{11}$ and the nilradical $\R^8$ is spanned by $e_{1},e_{2},e_{3},e_{4},e_{5},e_{6},e_{7},e_{8}$.


\subsection{$A_{4.4}$}

Lie brackets: $[e_1,e_4]=e_1, [e_2,e_4]=e_1+e_2, [e_3,e_4]=e_2+e_3$.

\noindent Geodesics:

\begin{equation}
\begin{array}{l}
 \ddot{x}=(\dot{x}+\dot{y})\dot{w} \\
 \ddot{y}=(\dot{y}+\dot{z})\dot{w}\\
 \ddot{z}=\dot{z}\dot{w} \\
 \ddot{w}=0.
\end{array}
\end{equation}

\noindent Lie symmetries:
\begin{eqnarray*}
&&e_1= D_x,\,\,\, e_2=D_y,\,\,\, e_3=D_z,\,\,\,
e_4=zD_x,\,\,\, e_5= yD_x+zD_y, \,\,\, e_6=e^wD_x,\,\,\,\\ && e_{7}=e^w(wD_x+D_y),\,\,\,
e_{8}=e^w(\frac{w^2}{2}D_x+wD_y+D_z),\\ && e_9=D_t,e_{10}=wD_t,
e_{11}=tD_t,\,\,\, e_{12}=D_w,\,\,\, e_{13}=x D_x+yD_y+zD_z.
\end{eqnarray*}

\noindent The symmetry algebra is a $13$-dimensional indecomposable solvable algebra $\R^3 \rtimes (N_8\oplus \R^2)$. Its nilradical is $10$-dimensional decomposable, a
direct sum of an $8$-dimensional nilpotent $N_8$ spanned by $e_{1},e_{2},e_{3},e_{4},e_{5},e_{6},e_{7},e_{8}$ and $\R^2$ spanned by $e_{9},e_{10}$. The complement to the
nilradical is abelian spanned by $e_{11},e_{12},e_{13}$.\\






\subsection{$A_{4.5ab}$}

\noindent Lie brackets: $[e_1,e_4]=e_1, [e_2,e_4]=ae_2, [e_3,e_4]=be_3,\,\, (ab\neq0, \, -1 \leq a \leq b \leq 1)$.

\noindent Geodesics:
\begin{equation}
\begin{array}{l}
 \ddot{x}=\dot{x}\dot{w} \\
 \ddot{y}=a\dot{y}\dot{w}\\
 \ddot{z}=b\dot{z}\dot{w} \\
 \ddot{w}=0.
\end{array}
\end{equation}

\noindent Generic case Lie symmetries:

\[ \begin{array}{l}
e_1= D_x,\,\,\, e_2=D_y,\,\,\, e_3=D_z,\,\,\,  e_4=e^wD_x, e_5=e^{aw}D_y,
e_6=e^{bw}D_z,\,\,\, e_7= D_t, \,\,\,\\
e_8=wD_t,\,\,\, e_{9}=xD_x,\,\,\, e_{10}=yD_y
e_{11}=zD_z,\,\,\, D_{12}=tD_t, \,\,\, e_{13}=D_w.\\
\end{array} \]\\


\noindent It is a $13$-dimensional indecomposable solvable Lie algebra $\R^5 \rtimes \R^8$. It has an $8$-dimensional abelian
nilradical spanned by $e_{1},e_{2},e_{3},e_{4},e_{5},e_{6},e_{7},e_{8}$ and $5$-dimensional abelian complement spanned by $e_{9},e_{10},e_{11},e_{12},e_{13}$.\\


\noindent $A_{4.5ab\, (a=1,b=1)}$\\
\noindent Lie symmetries:
\[ \begin{array}{l}
e_1= D_x,\,\,\, e_2=D_y,\,\,\, e_3=D_z,\,\,\,  e_4=e^wD_x, e_5=e^wD_y, \,\,\, e_6=e^wD_z,
e_7=D_t,\,\,\, e_8= wD_t,\\
e_9=xD_x+yD_y+zD_z,\,\,\, e_{10}=tD_t,\,\,\, e_{11}=D_w, \,\,\, e_{12}=yD_x
e_{13}=xD_y,\,\,\, e_{14}=zD_y\\
e_{15}= yD_z,\,\,\, e_{16}=xD_z,\,\,\, e_{17}= zD_x,\,\,\,
e_{18}= xD_x-zD_z,\,\,\, e_{19}=yD_y-zD_z.
\end{array} \]


\noindent It is a $19$-dimensional indecomposable Lie algebra with a non-trivial Levi decomposition. The semi-simple part is $\s\l(3,\R)$ and is spanned by $e_{12},e_{13},e_{14},e_{15},e_{16},
e_{17},e_{18},e_{19}$. The radical is a semi-direct product $\R^8 \rtimes\R^3$ with abelian
nilradical spanned by $e_{1},e_{2},e_{3},e_{4},e_{5},e_{6},e_{7},e_{8}$ and $3$-dimensional abelian complement spanned by $e_{9},e_{10},e_{11}$.
This finding is in agreement with Theorem \ref{thm}.\\


\noindent $A_{4.5ab\,(a=1,b=-1)}$\\
\noindent Lie symmetries:
\[ \begin{array}{l}
e_1= D_x,\,\,\, e_2=D_y,\,\,\, e_3=D_z,\,\,\,  e_4=e^wD_x, e_5=e^wD_y, \,\,\, e_6=e^{-w}D_z,
e_7=D_t,\,\,\, e_8= wD_t, \,\,\, e_9=tD_t,\\
e_{10}=3D_w+xD_x+yD_y-2zD_z,\,\,\, e_{11}=xD_x+yD_y+zD_z, \,\,\,
e_{12}=ze^wD_x,\,\,\, e_{13}=ze^wD_y,\\
e_{14}=xe^{-w}D_z,\,\,\,
e_{15}=ye^{-w}D_y,\,\,\, e_{16}= yD_x,\,\,\, e_{17}= xD_y,\,\,\,
e_{18}=xD_x-zD_z,\,\,\,e_{19}=yD_y-zD_z.\\
\end{array} \]
\noindent It is a $19$-dimensional indecomposable Lie algebra with a non-trivial Levi decomposition. The semi-simple part is $\s\l(3,\R)$ and is spanned by
$e_{12},e_{13},e_{14},e_{15},e_{16},
e_{17},e_{18},e_{19}$. The radical is a semi-direct product $\R^8 \rtimes\R^3$ with abelian
nilradical spanned by $e_{1},e_{2},e_{3},e_{4},e_{5},e_{6},e_{7},e_{8}$ and $3$-dimensional abelian complement spanned by $e_{9},e_{10},e_{11}$.\\


\noindent $A_{4.5ab\, (a=1)}$\\
\noindent Lie symmetries:
\[ \begin{array}{l}
e_1= D_x,\,\,\, e_2=D_y,\,\,\, e_3=z,\,\,\,  e_4=e^wD_x, e_5=e^wD_y, \,\,\, e_6=e^{bw}D_z,
e_7=D_t,\,\,\, e_8= wD_t, \,\,\,e_9=D_w, \,\,\,\\e_{10}= xD_x+yD_y,\,\,\, e_{11}=zD_z,\,\,\, e_{12}=tD_t, \,\,\, e_{13}=xD_x-yD_y,\,
e_{14}=yD_x,\,\,\, e_{15}=xD_y.\\
\end{array} \]
\noindent It is a $15$-dimensional indecomposable Lie algebra with a non-trivial Levi decomposition. The semi-simple part is $\s\l(2,\R)$ and is spanned by
$e_{13},e_{14},e_{15}$. The radical is a semi-direct product $\R^4 \rtimes\R^8$ with abelian
nilradical spanned by $e_{1},e_{2},e_{3},e_{4},e_{5},e_{6},e_{7},e_{8}$ and $4$-dimensional abelian complement spanned by $e_{9},e_{10},e_{11},e_{12}$.\\

\noindent $A_{4.5ab\, (a=-1)}$

\noindent Lie symmetries:
\[ \begin{array}{l}
e_1= D_x,\,\,\, e_2=D_y,\,\,\, e_3=D_z,\,\,\,  e_4=e^wD_x, e_5=e^{-w}D_y, \,\,\, e_6=e^{bw}D_z,\\

e_7=D_t,\,\,\, e_8= tD_t, \,\,\, e_9=wD_t,\,\,\, e_{10}=2D_w+xD_x-yD_y,\,\,\, e_{11}=xD_x+yD_y, \,\,\, e_{12}=zD_z\\

e_{13}=xD_x-yD_y,\,\,\, e_{14}=ye^wD_x,\,\,\,e_{15}= xe^{-w}D_y.\\
\end{array} \]


\noindent It is a $15$-dimensional indecomposable Lie algebra with a non-trivial Levi decomposition. The semi-simple part is $\s\l(2,\R)$ and is spanned by
$e_{13},e_{14},e_{15}$. The radical is a semi direct product $\R^4 \rtimes\R^8$ with abelian
nilradical spanned by $e_{1},e_{2},e_{3},e_{4},e_{5},e_{6},e_{7},e_{8}$ and $4$-dimensional abelian complement spanned by $e_{9},e_{10},e_{11},e_{12}$.\\


\subsection{$A_{4.6ab}$}

Lie brackets: $[e_1,e_4]=ae_1, [e_2,e_4]=be_2-e_3, [e_3,e_4]=e_2+be_3\,\,(a \neq 0, b \geq 0)$.

\noindent Geodesics:
\begin{equation}
\begin{array}{l}
 \ddot{x}=(b\dot{x}+\dot{y})\dot{w} \\
 \ddot{y}=(b\dot{y}-\dot{x})\dot{w}\\
 \ddot{z}=a\dot{z}\dot{w} \\
 \ddot{w}=0.
\end{array}
\end{equation}

\noindent Lie symmetries:
\[ \begin{array}{l}
e_1= D_t,\,\,\, e_2=D_x,\,\,\, e_3=D_y,\,\,\,  e_4=D_z, e_5=e^{aw}D_z, \,\,\, e_6=e^{bw}(\sin w\,D_x+\cos w\,D_y),\\
e_7=e^{bw}(-\cos w\,D_x+\sin w\,D_y),\,\,\, e_8= wD_t, \,\,\, e_9=zD_z,\,\,\, e_{10}=-yD_x+xD_y,\,\,\,\\
e_{11}=xD_x+yD_y,e_{12}=tD_t, e_{13}=D_w.\\
\end{array} \]


\noindent The symmetry algebra is a $13$-dimensional indecomposable solvable algebra $\R^5 \rtimes \R^8$. Its nilradical is $8$-dimensional abelian and
spanned by $e_{1},e_{2},e_{3},e_{4},e_{5},e_{6},e_{7},e_8$. The complement to the nilradical is abelian and spanned by $e_9, e_{10}, e_{11}, e_{12}, e_{13}$.\\


\noindent $A_{4.6ab, (b=0)}$

\noindent Lie symmetries:
\[ \begin{array}{l}
e_1= D_y,\,\,\, e_2=D_x,\,\,\, e_3=D_z,\,\,\,  e_4=D_w,\,\,\, e_5=tD_t, \,\,\, e_6=D_t,\\
e_7=zD_z,\,\,\, e_8= wD_t, \,\,\, e_9=xD_x+yD_y,\,\,\, e_{10}=-yD_x+xD_y,\,\,\,\\ e_{11}=e^{aw}D_z,
e_{12}=\sin w\,D_x+\cos w\,D_y,\,\,\,e_{13}=-\cos w\,D_x+\sin w\,D_y,\\
e_{14}=(y\cos w\,+x\sin w\,)D_x+(x\cos w\,-y\sin w)D_y,\\
e_{15}=-(x\cos w\,-y\sin w)D_x+(y\cos w\,+x\sin w)D_y.\\
\end{array} \]



\noindent The symmetry algebra is a $15$-dimensional indecomposable algebra that has a non-trivial Levi decomposition.\\
The semi-simple part is $\s\l(2,\R)$ and is spanned by
$e_{13},e_{14},e_{15}$. The nilradical is abelian and is
spanned\\ by $e_{1},e_{2},e_{3},e_{4},e_{5},e_{6},e_{7},e_8$ and the complement to the nilradical is abelian and spanned by $e_9, e_{10}, e_{11}, e_{12}$.
Again the ``extra" symmetry for $b=0$ is explained by the fact that, for such a value, eq.(\ref{newwww}) has a non-trivial solution.\\

\end{document}